\documentclass[10pt, reqno]{amsart}

\usepackage{amsmath}
\usepackage{amsfonts}
\usepackage{amssymb}

\usepackage{amssymb,latexsym,amsmath,amsfonts}
\usepackage{hyperref}
\usepackage{amsthm}
\usepackage{dsfont}
\usepackage[mathscr]{eucal}
\usepackage[english]{babel}
\theoremstyle{definition}
\newtheorem{definition}{Definition}
\newtheorem{thm}{Theorem}[section]
\newtheorem{lem}{ \bf Lemma}[section]
\newtheorem{prop}{\bf Proposition}[section]
\newtheorem{cor}{\bf Corollary}[section]
\theoremstyle{remark}

\newcommand{\be}{\begin{equation}}
\newcommand{\ee}{\end{equation}}
\newcommand{\Bea}{\begin{eqnarray*}}
\newcommand{\Eea}{\end{eqnarray*}}
\newcommand{\bea}{\begin{eqnarray}}
\newcommand{\eea}{\end{eqnarray}}

\numberwithin{equation}{section}

\def\de{{\delta}}

\begin{document}
\title[Uniform Poincar\'e Inequalities on measured metric spaces]{Uniform Poincar\'e inequalities on measured metric spaces}
\author{ Soma Maity \and Gautam Neelakantan M}

\address{Department of Mathematical Sciences, Indian Institute of Science Education and Research Mohali, \newline Sector 81, SAS Nagar, Punjab- 140306, India.}
\email{somamaity@iisermohali.ac.in}
\address{Department of Mathematical Sciences, Indian Institute of Science Education and Research Mohali, \newline Sector 81, SAS Nagar, Punjab- 140306, India.}
\email{neelanmemana@gmail.com}

\begin{abstract} Consider a proper geodesic metric space $(X,d)$ equipped with a Borel measure $\mu.$ We establish a family of uniform Poincar\'e inequalities on $(X,d,\mu)$ if it satisfies a local Poincar\'e inequality $P_{loc}$, and a condition on the growth of volume. Consequently, if $\mu$ is doubling and supports $P_{loc}$ then it satisfies a uniform $(\sigma,\beta,\sigma)$-Poincar\'e inequality. If $(X,d,\mu)$ is a Gromov-hyperbolic space, then using the volume comparison theorem in \cite{BCS}, we obtain a uniform Poincar\'e inequality with the exponential growth of the Poincar\'e constant. Next, we relate the growth of Poincar\'e constants to the growth of discrete subgroups of isometries of $X$, which act on it properly. We show that if $X$ is the universal cover of a compact $CD(K,\infty)$ space with $K\leq 0$, it supports a uniform Poincar\'e inequality, and the Poincar\'e constant depends on the growth of the fundamental group of the quotient space.
\end{abstract}

\footnote{Data sharing not applicable to this article as no datasets were generated or analysed during the current study.}
\maketitle

\section{Introduction}\label{sec1} Cheeger, Hajlasz, and Koskela showed the importance of local Poincar\'e inequalities in geometry and analysis on metric spaces with doubling measures in \cite{Ch}, \cite{HK}. In this paper, we establish a family of global Poincar\'e inequalities on geodesic spaces equipped with Borel measures, which satisfy a local Poincar\'e inequality along with certain other geometric conditions.

Let $(X,d)$ be a proper geodesic metric space, i.e., all closed balls in $(X,d)$ are compact and, any two points can be joined by a geodesic. Consider a Borel measure $\mu$ on $X$ such that every closed ball has a finite positive measure. We call the triple $(X,d,\mu)$ a measured metric space. A complete Riemannian manifold with distance and volume measure induced from the Riemannian metric is an example of a measured metric space. Riemannian manifolds with Ricci curvature bounded below may have polynomial or exponential growth of volume depending on bounds on curvature. Suppose there exists a non-decreasing function $f: (0,\infty)\to \mathbb{R}$ such that
\be \label{eqmain}\frac{\mu(B(x,R))}{\mu(B(x,\frac{1}{2}))}\leq f(R), \quad \ \forall x\in X , \ \forall \ R\geq \frac{1}{2}.
\ee  
This is a condition on the growth of volume on large scales. 
\begin{definition} Let $C(X)$ denote the space of continuous functions on $X$. An upper gradient of $u\in C(X)$ is a Borel function $g_u : X \to [0,\infty]$ such that for each curve $\gamma : [0, 1] \to X$ with finite length $l(\gamma)$ and constant speed,
$$|u(\gamma(1))-u(\gamma(0))|\leq l(\gamma)\int_0^1 g(\gamma(t)) dt.$$
Let $u_R$ denote the mean of $u$ on balls of radius $R$, i.e.,
$$ u_R(x) = \frac{1}{\mu(B(x,R))}\int_{B(x,R)} u \ d\mu .$$
\end{definition}
\begin{definition} Let $1 \leq \sigma < \infty$. $(X,d,\mu)$ is said to satisfy a local Poincar\'e inequality $P_{loc}$ if there exist positive constants $C(\sigma),r_0$ such that for every $u \in C(X)$ and its upper  gradient $g_u: X \to [0,\infty]$,
\be
	\int_{B(x,R)} |u - u_R|^{\sigma} d\mu \leq C \int_{B(x,R)} g_u^{\sigma} d\mu , \ \ \forall x\in X \ {\rm and} \ 0<R\leq r_0.   \label{local Poincare inequality}
 \ee
\end{definition} 
When the radius of the ball on the right-hand side is bigger than $R$, the inequality is called a weak Poincar\'e inequality. If $(X,d,\mu)$ supports a weak local Poincar\'e inequality then it satisfies a Poincar\'e inequality as above possibly with a different constant \cite{HK}, \cite{HK2}.
\begin{thm}\label{main} Let $(X,d,\mu)$ be a measured metric space which satisfies the growth condition (\ref{eqmain}) and $P_{loc}$ for $r_0\geq 1$ and $\sigma\geq 1.$ Then for any $u\in C(X)$ and its upper gradient $g_u$,
 	\be
 	\int_{B(x,R)} |u - u_R|^{\sigma} \ d\mu \leq C_0 R^{\sigma-1}f(4\lambda R)\int_{B(x,2\lambda R)}g_u^{\sigma} d\mu ,\quad  \forall \ R\geq 4\lambda
 	\ee
where $\lambda=f(8.5)+1 $ and $C_0=2^{4\sigma-2}\lambda^{3}f(12.5)C.$
 \end{thm}
The constant $C_0 R^{\sigma-1}f(4\lambda R)$ is called the Poincar\'e constant throughout this paper. It is an upper bound on the best constant for which the above inequality holds. In light of the Bishop-Gromov volume comparison theorem, it is interesting to study the growth of the Poincar\'e constant when $(X,d,\mu)$ satisfies a notion of a lower bound on Ricci curvature. 

Non-negative Ricci curvature on Riemannian manifolds plays a crucial role in establishing the existence of positive Green's functions and bounded harmonic functions, estimates on heat kernel and green functions, parabolic Harnack inequalities \cite{Li}. These results are also proved using variants of doubling measures and local Poincar\'e inequalities to replace non-negative Ricci curvature condition \cite{HK}, \cite{SC}. We refer to section \ref{sec4} for the definition of a doubling measure. If $\mu$ is a doubling measure on a measured metric space $(X,d,\mu)$, then the growth of volume is polynomial. If it also supports $P_{loc}$ then as a consequence of Theorem \ref{main} $(X,d,\mu)$ supports a uniform $(\sigma, \beta, \sigma)$-Poincar\'e inequality (see Corollary \ref{sigma beta sigma}), i.e., there exist positive constants $C_0, r, \beta$ and $\lambda\geq 1$ such that for any $u\in C(X)$ and its upper gradient $g_u$,
$$\int_{B(x,R)} |u - u_R|^{\sigma} \ d\mu \leq C_0 R^{\beta}\int_{B(x,\lambda R)}g_u^{\sigma} d\mu(z), \quad \forall \ x\in X, \ \forall R\geq r.$$
Some interesting examples of such metric spaces are complete Riemannian manifolds with non-negative Ricci curvature, finitely generated groups with polynomial growth, Lie groups with Carnot-Carath\'eodory metrics, topological manifolds with Ahlfors regular measures \cite{HK}. Besson, Courtois, and Hersonsky established a family of $(\sigma,\beta,\sigma)$-Poincar\'e inequality on complete Riemannian manifolds with Ricci curvature bounded below when the growth of volume is polynomial, and volumes of unit balls are bounded below by a positive constant in \cite{HCB}. However, by Theorem \ref{main}, for a Riemannian manifold with Ricci curvature bounded below it suffices to satisfy (\ref{eqmain}) for some polynomial $f(R)=vR^\alpha$ ($v>0$) to support a $(\sigma,\beta,\sigma)$-Poincar\'e inequality.

Next, we consider $\delta$-hyperbolic metric spaces in the sense of Gromov. The volume of a ball on a Riemannian manifold with negative sectional curvature grows exponentially as a function of its radius. Since $\delta$-hyperbolic spaces are defined generalizing certain metric properties of negatively curved Riemannian manifolds, it is natural to consider a Borel measure with exponential growth on them. The entropy of a measured metric space $(X,d,\mu)$ is defined by 
$${\rm Ent}(X,d,\mu)=\liminf_{R\to \infty}\frac{1}{R}\ln(\mu(B(x,R))).$$
It is independent of the choice of $x.$ Besson, Courtois, Gallot, and Sambusetti established a Bishop-Gromov volume comparison theorem on $\de$-hyperbolic measured metric spaces with volume entropy bounded above \cite{BCS}. They showed that if the measure is invariant under certain group action, then the growth of volume is exponential and, depends on the entropy of the space. They also pointed out that an upper bound on entropy may be considered as a lower bound on Ricci curvature in a weak sense. As a consequence of Theorem \ref{main} and the volume comparison theorem on $\delta$-hyperbolic spaces, we obtain the following theorem relating the growth of the Poincar\'e constant to the entropy of the space.
\begin{thm}\label{main3} Let $(X,d,\mu)$ be a measured $\de$-hyperbolic space which supports $P_{loc}$ for $r_0\geq 1$ and $\sigma\geq 1$. Let $\Gamma$ be a group acting on $X$ isometrically and properly such that the diameter of the quotient space $\Gamma \backslash X$ is bounded by $D$. Suppose the action of $\Gamma$ is also measure preserving and the entropy of $(X,d,\mu)$ is bounded by $H$. Then there exist $C_0(\delta,D,H,\mu, \sigma)>0$ and $\lambda \geq 1$ such that for any $u\in C(X)$ and its upper gradient $g_u$,
$$\int_{B(x,R)} |u - u_R|^{\sigma} d\mu \leq C_0 R^{\sigma+6HD+\frac{21}{4}}e^{12\lambda HR}\int_{B(x,2\lambda R)}g_u^{\sigma} d\mu , \ \ \forall R\geq \frac{5}{2}(7D+4\delta) , \ \forall \ x\in X.$$
\end{thm} 
For an explicit description of the constants, we refer to Theorem \ref{Phy} in Section \ref{sec4}. The above theorem generalizes the Poincar\'e inequality on Riemannian manifolds with Ricci curvature bounded below stated in section 10.1 in \cite{HK} on Gromov hyperbolic spaces. Cheeger showed that if a measured metric space admits $P_{loc}$, it has a nice local structure \cite{Ch}. Since on a $\delta$-hyperbolic space, local geometry and topology may be anything within a radius of $\delta$ of a point, $P_{loc}$ is a crucial assumption in the above theorem.  An upper bound of volume entropy only gives a bound on the growth of volume asymptotically. Hence a stronger notion of lower bound on Ricci curvature is required to replace $P_{loc}$ in the theorem above. If a group $\Gamma$ acts on a Riemannian manifold cocompactly, then it satisfies the $P_{loc}$ condition. K. Akutagawa, G. Carron, and R. Mazzeo showed that a large class of singular Riemannian manifolds, namely stratified spaces, also support $P_{loc}$ \cite{ACM}. Hence a $\delta$-hyperbolic stratified space admitting a group action as described in Theorem \ref{main3} satisfies a global Poincar\'e inequality with the exponential growth of Poincar\'e constant. We refer to Section 7 in \cite{BCS} for examples of such $\de$-hyperbolic spaces.

More generally, when a discrete group acts on a measured metric space, we study the growth of the Poincar\'e constant in terms of the growth of the group. Consider a discrete subgroup $\Gamma$ of isometries of $(X,d,\mu)$ acting on it properly such that the quotient space $\Gamma \backslash X$ is compact. Define,
\be
F_{\Gamma}(R)=|\Gamma x\cap \overline{B(x,R)}|.\label{growth of group}
\ee
Here $|.|$ denotes the cardinality of the set. $F_{\Gamma}(R)$ is independent of the choice of $x$, and it determines the growth of $\Gamma$ with respect to $R$.   Suppose the action of $\Gamma$ on $X$ is free, measure-preserving, and the quotient space $\Gamma \backslash X$ is compact. If the volume and the diameter of $\Gamma \backslash X$ are bounded above by $V$ and $D$, respectively, then
$$\mu(B(x,R))\leq VF_{\Gamma}(R+D), \quad \forall R\geq D.$$ 
When $\Gamma\backslash X$ supports a Poincar\'e inequality, we show that $(X,d,\mu)$ supports a $P_{loc}$ in Section \ref{sec5}. Then Theorem \ref{main} implies that $(X,d,\mu)$ admits a uniform Poincar\'e inequality (see Theorem \ref{PCov}).

A lower bound on Ricci curvature plays an important role in establishing a local Poincar\'e inequality on Riemannian manifolds \cite{Bu}, \cite{HK}, \cite{HCB}. Sturm, Lott, and Villani defined notions of lower bound on Ricci curvature on length spaces with probability measures in terms of optimal transports on Wasserstein spaces in the seminal papers \cite{Sturm1}, \cite{Sturm2}, \cite{LV1}. These are called $CD(K, N)$ spaces, where $K$ is a lower bound on Ricci curvature, and $N$ is an upper bound on the dimension. Lott and Villani proved Poincar\'e inequality on $CD(K, N)$ spaces, which are non-branching in \cite{LV1},\cite{LV2}. Later Rajala proved a local Poincar\'e inequality on $CD(K,\infty)$ spaces in the sense of Lott-Villani and Sturm without imposing the non-branching condition in \cite{Rajala} \cite{Rajala2}. We refer to \cite{Rajala} for the definition of $CD(K,\infty)$ spaces. 
\begin{thm} \cite{Rajala} \cite{Rajala2} Suppose that $(X,d,\mu)$ is a $CD(K,\infty)$ space with $K\leq 0$. Then for any continuous function $u$ on $X$ and for any upper gradient $g_u$ of $u$
\begin{align}
 	\int_{B(x,R)} |u-u_{B(x,R)}|\ d\mu \leq c(K,R) \int_{B(x,2R)} g_u\ d\mu, \ \ \forall R>0, \ \ \forall x\in X.
\end{align}
\end{thm}
From the proof of the above theorem, we observe the following result after applying Jensen's inequality.
\begin{prop}\label{Rajala} Suppose that $(X,d,\mu)$ is a $CD(K,\infty)$ space with $K\leq 0$. Then there exists a positive constant $c(K,\sigma,R)$  such that for any continuous function $u$ on $X$ and for any upper gradient $g_u$ of $u$
\bea
 	\int_{B(x,R)} |u-u_{B(x,R)}|^{\sigma}\ d\mu \leq c(K,\sigma,R) \int_{B(x,2R)} g_u^{\sigma}\ d\mu, \ \ \forall R>0, \ \ \forall x\in X, \ \sigma\geq 1.
\eea
$c(K,\sigma,R)$ is continuous in $R.$
\end{prop}
As a consequence of Theorem \ref{main}, we have the following theorem.
\begin{thm}\label{CDK} Let $\Gamma$ be a discrete subgroup of isometries of a measured metric space $(X,d,\mu)$ acting on it freely and properly such that the diameter of the quotient space $\Gamma\backslash X$ is bounded by $D$ . Suppose the action of $\Gamma$ is measure preserving and $\bar{d},\bar{\mu}$ denote the quotient metric and the quotient measure, respectively. If $(\Gamma\backslash X,\bar{d}, \bar{\mu})$ is a $CD(K,\infty)$ space with $K\leq 0$ then there exist positive constants $C(K,\sigma)$ such that for any $\sigma\geq 1,$ $u\in C(X)$ and its upper gradient $g_u$,
	\be
	\int_{B(x,R)} |u-u_R|^{\sigma} \ d\mu \leq 2^{4\sigma-2}CV_0\lambda^{4} R^{\sigma-1} F_{\Gamma}(2\lambda R)\int_{B(x,2\lambda R)} g_u^{\sigma} \ d\mu, \ \ \forall  R\geq r,\ x \in X.
	\ee 
where $V_0$ is the volume of $\Gamma\backslash X$ and $\lambda=V_0+D$.	
\end{thm}
Therefore, any covering space of a compact $CD(K,\infty)$ space with $K\leq 0$ satisfies a global Poincar\'e inequality as above. Moreover, if $X$ is simply connected, then the growth of the Poincar\'e constant depends on the growth of the fundamental group of the quotient space as described by Theorem \ref{CDK}. The space of all compact $CD(K,\infty)$ metric measured spaces is quite large. An interesting class of examples arises from differential geometry as Gromov-Hausdorff limits of compact Riemannian manifolds with Ricci curvature bounded below and limits of geometric flows if they exist. In particular if the quotient space $\Gamma\backslash X$ in Theorem \ref{main3} is a $CD(K,\infty)$ space, then $X$ satisfies  $P_{loc}$. \\
\\
{\bf Idea of proof and structure of the paper:} In this paper, the scheme of the proof of the existence of such a uniform Poincar\'e inequality is similar to the one used in \cite{HCB} and in \cite{SLT}. The authors showed that a complete Riemannian manifold $(M,g)$ with  polynomial growth of volume supports a uniform Poincar\'e inequality if and only if it satisfies a local Poincar\'e inequality and a graph approximation of $(M,g)$ supports a discrete version of Poincar\'e inequality \cite{HCB}, \cite{SLT}. By considering a more general growth condition ($\ref{eqmain}$) we are able to express the Poincar\'e constant in terms of the growth of volume, which allows us to apply the volume comparison theorem in various contexts. Techniques from the proof of classical Bishop-Gromov volume comparison theorem are used crucially to prove existing Poincar\'e inequalities on Riemannian manifolds. It was known that the growth of the Poincar\'e constant depends on the growth of volume. Theorem \ref{main} of this paper shows this dependency explicitly. The assumption on the lower bound on $r_0$ in Theorem \ref{main} is required to choose a graph discretization of $X$ with canonical combinatorial distance one. If $r_0$ is bounded below by a positive constant then the required bound may be achieved by scaling the metric $d$ suitably.

In Section \ref{sec2}, we establish a Poincar\'e inequality for a measured metric graph when it satisfies (\ref{eqmain}), which is an improvement of the Poincar\'e inequality established in \cite{HCB}. In \cite{HCB}, a weak Poincar\'e inequality for a measured metric graph is established under the assumption of polynomial growth of the measure of balls and a uniform lower bound on the measure of vertices, whereas, with an improvement in the proof, we obtain a strong Poincar\'e inequality without the extra assumption of uniform lower bound on the measure of vertices. Moreover, we show that the Poincar\'e inequality can be further improved if we assume a uniform lower bound on the measure of vertices. 

In Section \ref{sec3}, we first prove that an $\epsilon$-discretization of $(X,d,\mu)$, which is a measured metric graph, is roughly isometric to $(X,d,\mu)$ under the assumption of growth condition (\ref{eqmain}). Moreover, we obtain a growth function for the $\epsilon$-discretization satisfying (\ref{eqmain}) in terms of the growth function for $(X,d,\mu)$ satisfying (\ref{eqmain}). In this method of approximation by a graph, we try to emulate the foundational work of Kanai in \cite{Kanai}, \cite{Kanai2} and its later improvements made by Coulhon and Saloff-Coste in \cite{SLT}. Later in this section, with the assumption of the existence of a local Poincar\'e inequality on $(X,d,\mu)$ and the growth condition (\ref{eqmain}), we get a uniform Poincar\'e inequality from the Poincar\'e inequality established on its $\epsilon$-discretization which completes the proof of Theorem \ref{main}.  
 
In Section \ref{secpoly} we discuss examples of measured metric spaces with polynomial growth of volume and establish a $(\sigma,\beta,\sigma)$-type uniform Poincar\'e inequality. We study a family of uniform Poincar\'e inequalities on a Gromov hyperbolic space satisfying $P_{loc}$ with a bound on volume entropy and prove Theorem \ref{main3} in Section \ref{sec4}. The inspiration to consider Gromov $\delta$- hyperbolic spaces under such conditions is obtained from the volume comparison theorems proved in \cite{BCS}, which gives us an exponential growth function satisfying (\ref{eqmain}). Section \ref{sec5} is devoted to developing a relationship between the growth of Poincar\'e constants and the growth of groups. The existence of a Poincar\'e inequality on a covering space is also discussed when its quotient space admits a Poincar\'e inequality. We prove Theorem \ref{CDK} in this section.
\\
\\
{\bf Acknowledgement:} We sincerely thank G\'erard Besson and Gilles Courtois for introducing the Bishop-Gromov volume comparison theorem on Gromov hyperbolic spaces and its applications to us at the CIMPA school on Finsler Geometry at Varanasi. The first author is supported by the DST-INSPIRE faculty research grant, and the second author is supported by the Kishore Vaigyanik Protsahan Yojana fellowship.

\section{Poincar\'e Inequality on metric measured graphs}\label{sec2} Let $Y=(V,E)$ be a connected graph with a measure $\nu$, where $V,E$ denote the set of vertices and edges, respectively. We denote  $x\sim y$ if $x$ is adjacent to $y$. Define the distance  $\rho$ on $Y$ as the canonical combinatorial distance as follows:
 
$\rho(x,y)=1$ if and only if $x\sim y $. The length of a path $\gamma_{x,y}$ joining two points $x$ and $y$ is the number of edges in $\gamma_{x,y}$. Define $\rho(x,y)$ as the infimum of lengths of paths joining $x$ and $y.$ A graph $Y$ with a canonical combinatorial distance $\rho$ and a measure $\nu$,  is called a metric measured graph.
  
Let $u:V\to \mathbb{R}$ be a function. The integration with respect to a measure $\nu$ is defined as
\be
\int_Fu(x)d\nu(x)=\sum_{x\in F}u(x)\nu(x)  \ \ {\rm for} \ \ {\rm any} \ \ F\subset Y.
\ee
The point-wise $l^{\sigma}$-norm of the gradient of $u$ at a vertex $x$ is defined as
\be
|\delta u|_{\sigma}(x)=\left( \sum_{x\sim y}|u(x)-u(y)|^{\sigma} \right)^{\frac{1}{\sigma}}.
\ee
$L^{\sigma}$-norm of the gradient of $u$ with respect to $\nu$ over a $F\subset Y$ is 
\be
\|\delta u\|_{\sigma,F}^{\sigma}=\int_F|\delta u|_{\sigma}^{\sigma}d\nu .
\ee
Next, we establish a Poincar\'e inequality on a  metric measured graph $(Y,\rho,\mu)$ that satisfies a growth condition defined in (\ref{eqmain}). 
  
\begin{thm}\label{PGr} Let $(Y,\rho,\nu)$ be a metric measured graph and $f:(0,\infty)\to \mathbb{R}$ be a function such that  
$$\frac{\nu(B(x,R))}{\nu(x)}\leq f(R), \quad  \forall\ x\in Y \ {\rm and} \ R\geq r_0>0.$$
 Then for any $u:Y\to \mathbb{R}$, $\sigma\geq 1$ and $R\geq r_0>0,$ 
	$$\int_{B(p,R)}|u-u_R|^\sigma d\nu\leq 2^{\sigma} R^{\sigma-1}f(2R)\int_{B(p,R)}|\delta u|_{\sigma}^{\sigma}d\nu , \ \ \ \forall\ p\in Y.$$
\end{thm}
\begin{proof} Consider $u: Y\to \mathbb{R}$ and $R\geq r_0.$
	By applying Jensen's inequality, we have,
	\Bea \int_{B(p,R)}|u(x)-u_R|^\sigma d\nu(x)&\leq& \frac{1}{(\nu(B(p,R))^{\sigma}}\int_{B(p,R)}\left|\int_{B(p,R)}|u(x)-u(y)|d\nu(y)\right|^{\sigma}d\nu(x)\\
	&\leq& \frac{1}{\nu(B(p,R))}\int_{B(p,R)\times B(p,R)}|u(x)-u(y)|^{\sigma}d(\nu\otimes \nu).
	\Eea
	Minkowski inequality implies that
	\Bea \left( \int_{B(p,R)\times B(p,R)}|u(x)-u(y)|^{\sigma}d(\nu\otimes \nu)\right)^{\frac{1}{\sigma}}
	&\leq & \left(\int_{B(p,R)\times B(p,R)}|u(x)-u(p)|^{\sigma}d(\nu\otimes \nu)\right)^{\frac{1}{\sigma}}\\
	&&+
	\left(\int_{B(p,R)\times B(p,R)}|u(y)-u(p)|^{\sigma}d(\nu\otimes \nu)\right)^{\frac{1}{\sigma}}\\
	&=& 2\left(\int_{B(p,R)\times B(p,R)}|u(x)-u(p)|^{\sigma}d(\nu\otimes \nu)\right)^{\frac{1}{\sigma}}\\
	&=& \left(2^{\sigma}\nu(B(p,R))\int_{B(p,R)}|u(x)-u(p)|^{\sigma}d\nu(x)\right)^{\frac{1}{\sigma}}.\Eea
	Therefore,
	\be\label{G1}\int_{B(p,R)}|u(x)-u_R|^\sigma d\nu(x)\leq 2^{\sigma}\int_{B(p,R)}|u(x)-u(p)|^{\sigma}d\nu(x).\ee
	Let $\omega$ be the counting measure on $Y$. Let $\gamma_{p,x}=\{p=v_0,v_1,...,v_k=x\}$ be a minimal geodesic joining $p$ and $x$ for $x\in B(p,R).$ Then, using Jensen's inequality, 
	\Bea |u(x)-u(p)|^{\sigma} \leq k^{\sigma-1}\sum_{i=0}^{k-1} |u(v_{i+1})-u(v_i)|^{\sigma}\leq  k^{\sigma-1} \sum_{i=0}^{k-1}|\de u|_{\sigma}^{\sigma} (v_i)
	\leq  l_{p,x}^{\sigma-1}\int_{\gamma_{p,x}}|\de u|_{\sigma}^{\sigma}d\omega
	\Eea
	where $l_{p,x}=$ length$(\gamma_{p,x}).$ Since $l_{p,x}\leq R$ we have,
	\bea\label{G2} |u(x)-u(p)|^{\sigma}\leq  R^{\sigma-1}\int_{\gamma_{p,x}}|\de u|_{\sigma}^\sigma d\omega\leq R^{\sigma-1} \int_{B(p,R)}|\de u|_{\sigma}^{\sigma}(y) d\omega(y).
	\eea
Now,
	\bea\label{G3} \int_{B(p,R)}|\de u|_{\sigma}^{\sigma}(y) d\omega(y) &=&\sum_{y\in B(p,R)} |\de u|_{\sigma}^{\sigma}(y)= \frac{1}{\nu(B(p,R))}\sum_{y\in B(p,R)} |\de u|_{\sigma}^{\sigma}\nu(y)\frac{\nu(B(p,R))}{\nu(y)}.
	\eea
For any $y\in B(p,R)$, $B(p,R)\subset B(y,2R).$  Therefore,
	\Bea \frac{\nu(B(p,R))}{\nu(y)}\leq \frac{\nu(B(y,2R))}{\nu(y)}\leq f(2R).
	\Eea
	Hence from (\ref{G3}), we have,
	\Bea \int_{B(p,R)}|\de u|_{\sigma}^{\sigma}(y) d\omega(y) \leq \frac{f(2R)}{\nu(B(p,R))}\int_{B(p,R)}|\de u|_{\sigma}^{\sigma}(y) d\nu(y).
	\Eea
	Combining (\ref{G1}), (\ref{G2}) and the above inequality, we have,
	\Bea \int_{B(p,R)}|u(x)-u_R|^\sigma d\nu(y)&&\leq  \int_{B(p,R)}\left(\frac{2^{\sigma}R^{\sigma-1}f(2R)}{\nu(B(p,R))}\int_{B(p,R)}|\de u|_{\sigma}^{\sigma}(y)d\nu(y)\right)d\nu(x)\\
	&&\leq 2^{\sigma}R^{\sigma -1}f(2R)\int_{B(p,R)}|\de u|_{\sigma}^{\sigma}(y)d\nu(y).
	\Eea
\end{proof}
If there exists $c>0$ such that  $\mu(x)\geq \frac{1}{c}$ for all $x\in X$ then we have the following theorem.
\begin{thm}\label{PGr2} Let $(Y,\rho,\nu)$ be a metric measured space. Suppose there exists a constant $c>0$ and a function $f:(0,\infty)\to \mathbb{R}$ such that 
$$\nu(B(x,R))\leq f(R) \ \ {\rm and} \ \ \nu(x)\geq \frac{1}{c}, \quad \forall x\in Y \ {\rm and} \ R\geq r_0>0.$$ 
Then for any $u:Y\to \mathbb{R}$, $\sigma\geq 1$ and $R\geq r_0>0,$ 
	$$\int_{B(p,R)}|u-u_R|^\sigma d\nu\leq 2^{\sigma}c R^{\sigma-1}f(R)\int_{B(p,R)}|\delta u|_{\sigma}^{\sigma}d\nu , \ \ \ \forall\ p\in Y.$$
\end{thm}
\begin{proof} From (\ref{G3}) we have,
	\bea \int_{B(p,R)}|\de u|_{\sigma}^{\sigma} d\omega =\sum_{y\in B(p,R)} |\de u|_{\sigma}^{\sigma}(y)
	\leq c\sum_{y\in B(p,R)} |\de u|_{\sigma}^{\sigma}(y)\nu(y)
\leq c\int_{B(p,R)}|\de u|_{\sigma}^{\sigma}(y) d\nu(y).
	\eea
	From (\ref{G1}) and (\ref{G2}) we have,
	\Bea \int_{B(p,R)}|u(x)-u_R|^\sigma d\nu(y)&&\leq  \int_{B(p,R)}\left(2^{\sigma}cR^{\sigma-1}\int_{B(p,R)}|\de u|_{\sigma}^{\sigma}(y)d\nu(y)\right)d\nu(x)\\
	&&\leq 2^{\sigma}cR^{\sigma -1}\nu(B(p,R))\int_{B(p,R)}|\de u|_{\sigma}^{\sigma}(y)d\nu(y)\\
	&&\leq2^{\sigma}cR^{\sigma -1}f(R)\int_{B(p,R)}|\de u|_{\sigma}^{\sigma}(y)d\nu(y).
	\Eea
\end{proof}
Since $f$ is an increasing function, $f(2R)\geq f(R)$. Hence the Poincar\'e constant here is slightly better than that of the previous theorem for sufficiently large $R.$ In Theorem 4.2 in \cite{HCB}, the authors established a weak Poincar\'e inequality on a metric measured graph that satisfies a polynomial growth of measure and a uniform lower bound on $\nu(x)$. Theorem \ref{PGr2} and Theorem \ref{PGr} improve the Poincar\'e constant in Theorem 4.2 in \cite{HCB} in the case of polynomial growth. 
\section{Uniform Poincar\'e inequalities on metric spaces}\label{sec3}
In this section, we prove the main theorem. Consider a geodesic measured metric space $(X,d,\mu).$ 
\begin{definition}\label{discretization} A  graph $Y$ with a metric $\rho$ and a measure $\mu $ on it,  is said to be an $\epsilon$-discretization of $(X,d,\mu)$ for any $\epsilon>0$ if  the following conditions hold :

(i) $Y$ is a maximal $\epsilon$-separated set in $X$ i.e. $d(y_i,y_j)\geq \epsilon ,\ \forall y_i, y_j\in Y$ with $y_i\neq y_j$.

(ii) $\rho(y_i,y_j)=\epsilon$ for any $y_i,y_j\in Y$ if  $d(y_i,y_j)<2\epsilon$ and  $y_i\neq y_j$  ;  $\rho(y_i,y_j)=0 $ if and only if $y_i=y_j.$

(iii) $\nu(y)=\mu(B_X(y,\epsilon)), \quad \forall y\in Y.$

(iv) Define a graph with $Y$ as the set of vertices. Any $y_i\sim y_j$ if $\rho(y_i,y_j)=\epsilon$ for all $y_i,y_j\in Y.$
\end{definition}
Observe that, given $\epsilon>0$, $(X,d,\mu)$ admits such an $\epsilon$-discretization by Zorn's lemma and $\{B_X(y,\epsilon):y\in Y\}$ covers $X.$ For any $L\geq 1$, the multiplicity of the covering $\{(B_X(y,L\epsilon))\}_{y\in Y}$ is defined as 
$$\mathcal{M}(Y,L\epsilon)=\sup_{y\in Y} |\{z\in Y: B_X(y,L\epsilon)\cap B_X(z,L\epsilon)\neq \phi \} |.$$
We obtain an estimate of the multiplicity of an $\epsilon$-discretization of $X$ in terms of the growth function $f$ following similar steps as in \cite{HCB}.  
\begin{lem}\label{Bound on multiplicity} Let $(Y,\rho,\nu)$ be an $\epsilon$-discretization of $(X,d,\mu)$ with $\epsilon\geq 1$. If $X$ satisfies (\ref{eqmain}) and $\mathcal{M}(Y,L\epsilon)$ denotes the multiplicity of the covering $\{B(x,L\epsilon)\}_{x\in Y}$ then
$$\mathcal{M}(Y, L\epsilon)\leq f(4L\epsilon+\frac{1}{2}) , \ \ \forall L \geq 1.$$
	\end{lem}
\begin{proof}
Every ball considered in this lemma are with respect to the distance $d$ in $X$. Since $\epsilon\geq 1$ observe that $\{B(x,\frac{1}{2})\}_{x\in Y}$ is a disjoint family of balls in $X$. Now, consider the set $Z \subset Y$ such that for all $z\in Z$,  $B(x,L\epsilon) \cap B(z, L\epsilon) \neq \emptyset$ for some fixed $x\in Y$.   Hence $Z \subset B(x, 2L\epsilon)$ and $\{B(z,\frac{1}{2})\}_{z\in Z} $ is a disjoint family of balls contained in $B(x,2L\epsilon +\frac{1}{2})$. Hence,
	\bea\sum_{z\in Z} \mu(B(z,\frac{1}{2})) \leq \mu(B(x,2L\epsilon+\frac{1}{2})).  \label{Bound on multiplicity 1}
	\eea
	Then, by using the  condition on the growth of volume of balls
	$$ \mu(B(z,4L\epsilon+\frac{1}{2})) \leq f(4L\epsilon+\frac{1}{2}) \mu(B(z,\frac{1}{2})), \ \ \forall z\in Y.
	$$
	Now by using the above equation and the fact that for $z\in Z$, $B(x,2L\epsilon+\frac{1}{2}) \subset B(z,4L\epsilon+\frac{1}{2})$ we get
	\bea \label{Bound on multiplicity 2}
	\sum_{z\in Z}\mu(B(z,\frac{1}{2})) \geq \frac{|Z|}{f(4L\epsilon+\frac{1}{2})}\mu(B(x,2L\epsilon+\frac{1}{2})).
	\eea
Now, combining equations (\ref{Bound on multiplicity 1}) and (\ref{Bound on multiplicity 2}) we get $|Z|\leq f(4L\epsilon+\frac{1}{2})$. Hence the lemma follows.
\end{proof}
In next two lemmas we establish relations between distances and growth of measures of $(X,d,\mu)$ and $(Y,\rho,\nu)$ in terms of the growth function $f$. We used ideas from similar lemmas proved in \cite{Kanai2}, and \cite{SLT} on Riemannian manifolds. 
\begin{lem}\label{qi} Let $(Y,\rho,\nu)$ be an $\epsilon$-discretization of $(X,d,\mu)$ with $\epsilon\geq 1$. If $X$ satisfies (\ref{eqmain}) then
$$d(x,y)\leq 2\rho(x,y) \leq 2f(8\epsilon+\frac{1}{2})(d(x,y) +2\epsilon) \ \ \forall  x,y \in Y .$$ 
\end{lem}
\begin{proof} The first inequality follows easily from triangle inequality and Definition \ref{discretization}. To prove the second inequality, consider  a geodesic  $\gamma$ in $X$ joining $x$ and $y$. Let $Y_{\gamma} = \{z \in Y : B(z,\epsilon) \cap \gamma \neq \emptyset\}$. Clearly, $\{B(z,\epsilon): z\in Y_{\gamma}\}$ covers $\gamma$ and $\rho(x,y) \leq \epsilon\ \ |Y_{\gamma}|$. Consider the positive integer $k$ such that $k-1 < d(x,y)/\epsilon \leq k$. Let $(x=x_0,x_1,..,x_{k-1},x_k=y)$ be points on $\gamma$ such that $d(x_{j-1},x_j)=d(x,y)/k$ for $j=1,...k$. Since $Y_{\gamma}$ is contained in an $\epsilon$ neighbourhood of $\gamma$, $Y_{\gamma} \subset \cup_{j=0}^{k}\{z\in Y : x_j \in B(z,2\epsilon)\}$. By Lemma \ref{Bound on multiplicity}, 
$$\rho(x,y) \leq \epsilon\  |Y_{\gamma}|\leq \epsilon \sum_{j=0}^{k} |\{z\in Y : x_j \in B(z,2\epsilon) \}| \leq \epsilon (k+1) f(8\epsilon +\frac{1}{2}) < f(8\epsilon+\frac{1}{2}) \big(d(x,y) +2\epsilon \big).$$
 Hence we have the required inequality.
\end{proof}
\begin{lem}\label{ri} Let $(Y,\rho,\nu)$ be an $\epsilon$-discretization of $(X,d,\mu)$ with $\epsilon\geq 1$. If $X$ satisfies (\ref{eqmain}),  then for all $x\in Y$
\be
 \nu(B_Y(x,R))\leq f(\epsilon)\mu(B_X(x,2R+\frac{1}{2}))\label {eq th 2.20}
\ \ \ \
{\rm and} \ \ \ \
 \frac{\nu(B_Y(x,R))}{\nu(B_Y(x,\frac{1}{2}))}\leq f(\epsilon)f(2R+\frac{1}{2}).
\ee
If $R'=f(8\epsilon+\frac{1}{2})(R+3\epsilon)$ then
$$\mu(B_X(x,R))\leq \nu(B_Y(x,R')) \ {\rm and} \ \frac{\mu(B_X(x,R))}{\mu(B_X(x,\frac{1}{2}))}\leq f(\epsilon)\frac{\nu(B_Y(x,R'))}{\nu(B_Y(x,\frac{1}{2}))}.$$
\end{lem}
\begin{proof} 
$$\nu(B_Y(x,R))=\sum_{y\in B_Y(x,R)}\nu(y)=\sum_{y\in B_Y(x,R)}\mu(B_X(y,\epsilon))=f(\epsilon)\sum_{y\in B_Y(x,R)}\mu(B_X(y,\frac{1}{2})).$$
From Lemma \ref{qi} we obtain that $y\in B_X(x,2R)$ for all $y\in B_Y(x,R)$. Observe that $\{B_X(y,\frac{1}{2})\}_{y\in Y}$ are mutually disjoint and contained in $B_X(x, 2R+\frac{1}{2})$. Hence,
$$\nu(B_Y(x,R))\leq f(\epsilon)\mu(B_X(x,2R+\frac{1}{2})).$$
Using the above inequality and (\ref{eqmain}), we get that
$$\frac{\nu (B_Y(x,R))}{\nu(x)}\leq f(\epsilon)\frac{\mu(B_X(x,2R+\frac{1}{2}))}{\mu(B_X(x,\epsilon))}\leq f(\epsilon)f(2R+\frac{1}{2}).$$
Let $R'=f(8\epsilon+\frac{1}{2})(R+3\epsilon)$ and $y\in B_X(x,R+\epsilon)\cap Y.$ By Lemma \ref{qi} ,  $y\in B_Y(x, R')$. Hence,
$$ B_X(x,R)\subset \cup_{y\in B_X(x,R+\epsilon)} B_X(y,\epsilon) \subset \cup_{y\in B_Y(x,R')}B_X(y,\epsilon).$$
This implies
$$\mu(B_X(x,R))\leq \sum_{y\in B_Y(x,R')}\mu(B_X(y,\epsilon))=
\nu (B_Y(x,R')).$$
Therefore,
$$\frac{\mu(B_X(x,R))}{\mu(B_X(x,\frac{1}{2}))}\leq f(\epsilon)\frac{\mu(B_X(x,R))}{\mu(B_X(x,\epsilon))}\leq f(\epsilon)\frac{\nu(B_Y(x,R'))}{\nu(B_Y(x,\frac{1}{2}))}.$$
\end{proof}
Next,  we define discretization of a continuous function on $X.$
For any $u\in C(X)$, $\tilde{u} : Y \to \mathbb{R}$ is given by
$$
\tilde{u} (x) = u_{B(x,\epsilon)\cap Y} = \frac{1}{\mu(B(x,\epsilon))} \int_{B(x,\epsilon)} u(z)\ d\mu(z).
$$ 
Let $||u||_{\sigma,E}$ denote the $L^{\sigma}$-norm of a Borel function $u$ on a Borel set $E$  for any $\sigma \geq 1$ with respect to $\mu$. 
Lemma 3.5 in \cite{HCB} relates the $L^\sigma$-norm of $\delta \tilde{u}$ and the $L^{\sigma}$-norm of the gradient of $u$ in the case of Riemannian manifolds. Following the same idea, we prove the following lemma.
\begin{lem}\label{important lemma} Let  $(Y,\rho,\nu)$ be an $\epsilon$-discretization of $(X,d,\mu)$ with $\epsilon\geq 1$ and let $g_u$ be an upper gradient of $u\in C(X)$. If $X$ satisfies (\ref{eqmain}) and the $P_{loc}$ condition in (\ref{local Poincare inequality}) for $r_0\geq \epsilon$ then $\forall x\in Y$ and $R\geq r_0$,
$$||\delta{\tilde{u}}||^{\sigma}_{\sigma,B(x,R)\cap Y} \leq 2^{\sigma-1}C(1+f(4\epsilon+\frac{1}{2}))^2f(12\epsilon+\frac{1}{2})||g_u||^{\sigma}_{\sigma,B(x,R+3\epsilon)}.$$
  
\end{lem}
\begin{proof}
Let $\epsilon>0$ and $z,z'\in X$ such that $d(z,z')<2\epsilon.$ Then
$$\tilde{u}(z)-\tilde{u}(z')=\frac{1}{\mu(B(z,\epsilon))}\frac{1}{\mu(B(z',\epsilon))}\int_{B(z,\epsilon)}\int_{B(z',\epsilon)}(u(s)-u(t))d\mu(s)d\mu(t).$$
Using Jensen's inequality, we have,
$$|\tilde{u}(z)-\tilde{u}(z')|^{\sigma}\leq\frac{1}{\mu(B(z,\epsilon))}\frac{1}{\mu(B(z',\epsilon))}\int_{B(z,\epsilon)}\int_{B(z',\epsilon)}|u(s)-u(t)|^{\sigma}d\mu(s)d\mu(t).$$
Now using Minkowski's inequality and then $P_{loc}$, we get,
\Bea|\tilde{u}(z)-\tilde{u}(z')|^{\sigma}&\leq&2^{\sigma-1}\frac{1}{\mu(B(z,\epsilon))}\frac{1}{\mu(B(z',\epsilon))}\\
 &\times & \int_{B(z,\epsilon)}\int_{B(z',\epsilon)}(|u(s)-\tilde{u}(z)|^{\sigma}+ |u(t)-\tilde{u}(z')|^{\sigma}d\mu(s)d\mu(t)\\
&\leq& \frac{2^{\sigma-1}C}{\mu(B(z,\epsilon))}\int_{B(z,\epsilon)}g_u^\sigma
+  \frac{2^{\sigma-1}C}{\mu(B(z',\epsilon))}\int_{B(z',\epsilon)}g_u^{\sigma} .\Eea
Since $d(z,z')<2\epsilon$, $B(z',\epsilon)\subset B(z,3\epsilon)$. Therefore,
\Bea |\tilde{u}(z)-\tilde{u}(z')|^{\sigma}&\leq&\frac{2^{\sigma-1}C}{\mu(B(z,\epsilon))}\left(1+\frac{\mu(B(z,\epsilon))}{\mu(B(z',\epsilon))}\right)\int_{B(z,3\epsilon)}g_u^\sigma.
\Eea
As $B(z,\epsilon)\subset B(z',3\epsilon)$ we have 
\Bea |\tilde{u}(z)-\tilde{u}(z')|^{\sigma}\mu(B(z,\epsilon))\leq 2^{\sigma-1}C\left(1+\frac{\mu(B(z',3\epsilon))}{\mu(B(z',\epsilon))}\right)\int_{B(z,3\epsilon)}g_u^\sigma.
\Eea
The growth of volume condition implies,
\bea\label{imle2}\nonumber |\tilde{u}(z)-\tilde{u}(z')|^{\sigma}\mu(B(z,\epsilon))
&\leq & 2^{\sigma-1}C\left(1+f(3\epsilon)\right)\int_{B(z,3\epsilon)}g_u^\sigma d\mu\\
&\leq &  2^{\sigma-1}C(1+f(3\epsilon))\int_{B(x,R+3\epsilon)}g_u^\sigma \chi_{B(z,3\epsilon)}d\mu
\eea
where $\chi_S$ denotes the characteristic function of a set $S.$ Let 
$$B(x,R)\cap Y=\{z_{\alpha}:\alpha\in \Lambda \} \ \ {\rm and} \ \ g_{u\alpha}=g_u.\chi_{B(z_\alpha,3\epsilon)}.$$
Therefore,
\Bea && \|\delta \tilde{u}\|^{\sigma}_{\sigma,B(x,R)\cap Y}=\sum_{\alpha\in \Lambda}|\delta \tilde{u}|_{\sigma}^{\sigma}(z_{\alpha})\mu(B(z_{\alpha},\epsilon))\\
&\leq & \sum_{\alpha\in \Lambda}\sum_{z_{\beta}\in B(z_\alpha,2\epsilon)}|\tilde{u}(z_{\beta})-\tilde{u}(z_{\alpha})|^{\sigma}\mu(B(z_{\alpha},\epsilon))\\
&\leq & 2^{\sigma-1}C(1+f(3\epsilon))\mathcal{M}(Y,\epsilon)\sum_{\alpha\in \Lambda}\int_{B(x,R+3\epsilon)}|g_{u\alpha}|^{\sigma}d\mu \\
&\leq & 2^{\sigma-1}C(1+f(3\epsilon))\mathcal{M}(Y,\epsilon)\int_{B(x,R+3\epsilon)}\sum_{\alpha\in \Lambda}|g_{u\alpha}|^{\sigma}d\mu.
\Eea 
The third line follows from (\ref{imle2}). Now for any fixed $z\in B(x,R+3\epsilon)$, $g_{u\alpha}(z)$ is non-zero only if $z\in B(z_{\alpha},3\epsilon)$ for some $\alpha\in \Lambda$. Hence, 
\Bea \sum_{\alpha\in \Lambda}|g_{u\alpha}|^{\sigma}(z) &\leq & \mathcal{M}(Y,3\epsilon)g_u^{\sigma}(z)\quad \forall z\in B(x,R+3\epsilon). 
\Eea
Therefore,
\Bea && \|\delta \tilde{u}\|^{\sigma}_{\sigma,B(x,R)\cap Y}\leq 2^{\sigma-1}C(1+f(3\epsilon))\mathcal{M}(Y,\epsilon)\mathcal{M}(Y,3\epsilon)\int_{B(x,R+3\epsilon)}g_u^\sigma d\mu.
\Eea
The required result follows from Lemma \ref{Bound on multiplicity}.
\end{proof}
Next, we prove the main theorem.
\subsection*{Proof of Theorem \ref{main} :}
\begin{proof} Let $(X,d,\mu)$ be a measured metric space which satisfies $P_{loc}$ and the growth condition as defined in (\ref{local Poincare inequality}) and (\ref{eqmain}) respectively. Fix $R>0$. Let $(Y,\rho,\nu)$ be a fixed $\epsilon$-discretization of $(X,d,\mu)$ with $\epsilon = 1$.  Since $B(x,R) \subset \bigcup _{y \in Y \cap B(x,R+1)} B(y,1)$, for any $\eta\in \mathbb{R}$ we have,
\Bea
\int_{B(x,R)} |u - \eta| ^{\sigma} \ d\mu
&\leq &\sum_{y\in Y \cap B(x,R+1)} \int_{B(y,1)} |u-\eta|^{\sigma} \ d\mu.
\Eea
By applying Jensen's inequality, we have,
\bea\label{eta} \nonumber \int_{B(x,R)} |u- \eta|^{\sigma} \ d\mu
&\leq & 2^{\sigma-1} \sum_{y\in Y\cap B(x,R+1)} \int_{B(y,1)} |u-\tilde{u}(y)|^{\sigma} \ d\mu\\
&&+ 2^{\sigma -1} \sum_{y\in Y \cap B(x,R+1)} \nu(y) |\tilde{u} (y) -\eta|^{\sigma}.
\eea
Let us denote by $(I)$ and $(II)$, the first and the second term of the right-hand side of the last inequality, respectively. One can bound $(I)$ by local Poincar\'e inequality (\ref{local Poincare inequality}) for radius $1$ since $r_0\geq 1$.
\Bea
(I) \leq 2^{\sigma -1}C \sum_{y \in Y \cap B(x,R+1)} \int_{B(y,1)} g_u ^{\sigma} \ d\mu.
\Eea
Using the same argument as in Lemma \ref{important lemma} we have,
\bea
\nonumber \label{I} (I) &\leq & 2^{\sigma -1}C\mathcal{M}(Y,1) \int_{B(x,R+2)} g_u ^{\sigma} \ d\mu\\
&\leq & 2^{\sigma -1}Cf(4.5) \int_{B(x,R+2)} g_{u} ^{\sigma} \ d\mu.
\eea 
We obtain a bound on $(II)$ for a certain value of $\eta$ using the Poincar\'e inequality on $(Y,\rho,\nu)$. We choose $x_0 \in Y$ such that $d(x,x_0) <1$. Hence,
$$
Y \cap B(x,R+1) \subset B(x_0,R+2).
$$
To apply the Poincar\'e inequality on $Y$, consider $r=f(8.5)(R+4)$, $h(r)=f(1)f(2r+\frac{1}{2})$. Then using Lemma \ref{qi} and Lemma \ref{ri} we have, 
$B_X(x_0,R+2)\cap Y\subset B_Y(x_0,r)$ and $\frac{\nu(r)}{\nu(\frac{1}{2})}\leq h(r).$ Let
 $$\tilde{u}_r = \frac{1}{\nu(B(x_0,r))} \sum_{y\in B_Y(x_0,r)} \tilde{u}(y) \nu(y).
$$
 Next we choose $\eta=\tilde{u}_r$ and estimate $(II)$ for this particular value of $\eta.$ From Theorem \ref{PGr} we obtain,
$$
(II) \leq 2^{2\sigma-1}  r^{\sigma-1}h(2r) \sum_{y\in B_Y(x_0,r)} |\delta \tilde{u}(y)|_{\sigma}^{\sigma} \nu (y). 
$$
As $B_Y(x_0,r)\subset B_X(x_0,2r)$ from Lemma \ref{qi}, we have,
$$(II) \leq 2^{2\sigma-1}  r^{\sigma-1}h(2r) \sum_{y\in B_X(x_0,2r)\cap Y} |\delta \tilde{u}(y)|_{\sigma}^{\sigma} \nu (y). 
$$
Now, by Lemma \ref{important lemma} we have,
\Bea
(II) &\leq & 2^{3\sigma-2} r^{\sigma-1}C(1+f(4.5))^2f(12.5)f(1)f(4r+0.5)  \int_{B(x_0, 2r+3)} g_u^{\sigma} \ d\mu\\
&\leq & 2^{3\sigma-2} r^{\sigma-1}C(1+f(4.5))^2f(4.5)f(12.5)f(4r+0.5)  \int_{B(x, 2r+4)} g_u^{\sigma} \ d\mu.
\Eea
Let $\lambda=f(8.5)+1$. Then for all $R\geq 1$ and $\eta=\tilde{u}_r$, we have from (\ref{I}), 
$$(I)+(II)\leq 2^{3\sigma-2} r^{\sigma-1}C\lambda^{3}f(12.5)f(4r+0.5)  \int_{B(x, 2r+4)} g_u^{\sigma} \ d\mu.$$
Assuming $R\geq 4\lambda$ we have,
\bea \label{II}
\int_{B(x,R)} |u- \tilde{u}_r| ^{\sigma} \ d\mu\leq 2^{3\sigma-2}\lambda^{3} f(12.5)CR^{\sigma-1}f(4\lambda R)  \int_{B(x, 2\lambda R)} g_u^{\sigma} \ d\mu.
\eea
We have the required uniform Poincar\'e inequality from the following.
\be
\int_{B(x,R)} |u-u_R|^{\sigma} \ d\mu\leq 2^{\sigma} \underset{\tau \in \mathbb{R}}{inf} \int_{B(x,R)} |u -\tau|^{\sigma} \ d\mu\leq \int_{B(x,R)} |u- \tilde{u}_r| ^{\sigma} \ d\mu.
\ee
We refer to \cite{HCB} for proof of the above inequality.
\end{proof}
Moreover, if the measure of every ball of radius $\frac{1}{2}$ is bounded away from zero, then, with some additional assumption on the growth of the measure, we can improve the constants in Theorem \ref{main}.
\begin{thm}\label{main2}
Let $(X,d,\mu)$ be a measured metric space which satisfies $P_{loc}$ for some $r_0\geq 1$, $\sigma \geq 1$ as in (\ref{local Poincare inequality}). Suppose there exist a non-decreasing function $V:(0,\infty)\to \mathbb{R}$  such that 
\be \label{grvol2} \mu(B(x,R))\leq V(R) \ \ {\rm and } \ \  \mu(B(x,\frac{1}{2}))\geq 1, \quad \forall x\in X, \ \ \forall R>0. \ee
Then for any $u\in C(X)$ and its upper gradient $g_u$,
	\be
	\int_{B(x,R)} |u - u_R|^{\sigma}\leq 2^{4\sigma-2}C\lambda^{4} R^{\sigma-1}V(2\lambda R)\int_{B(x,2\lambda R)}g_u^{\sigma}d\mu\ee
	for all $R\geq  4\lambda$ where $\lambda=\max\{V(6.5),V(4.5)+1\}.$
\end{thm}
\begin{proof} Let $(Y,\rho,\nu)$ be an $\epsilon$-discretization of $(X,d,\mu)$ with $\epsilon= 1.$ From the inequality (\ref{Bound on multiplicity 1}) we obtain the multiplicity of the covering $\{B(y,L)\}_{y\in Y}$ for any $L\geq 1$ as follows.
\be\label{BM} \mathcal{M}(Y,L)\leq V(2L+\frac{1}{2}).
\ee
Using the growth of volume (\ref{grvol2}) we obtain the following inequalities from the proof of Lemma \ref{qi} and Lemma \ref{ri} respectively.
\be\label{ri2} \rho(x,y)\leq V(4.5)(d(x,y)+2)\ee
and 
\be \label{qi2} \nu(B_Y(x,R))\leq V(1)V(2R).\ee
Using (\ref{BM}) in Lemma \ref{important lemma}, we obtain 
\be\label{IL} ||\delta{\tilde{u}}||^{\sigma}_{\sigma,B_X(x,R)\cap Y} \leq 2^{\sigma-1}C(1+V(3))^{2}V(6.5)||g_u||^{\sigma}_{\sigma,B(x,R+3)}. 
\ee
For $u\in C(X),$ let $g_u$ be an upper gradient of $u.$ From (\ref{eta}) we have for any $\eta>0$, 
$$\int_{B(x,R)}|u(z)-\eta|^{\sigma}\leq (I)+(II)$$
where $(I)$ and $(II)$ are the first term and the second term in (\ref{eta}), respectively. From (\ref{I}) we have,
\be \label{I2} I\leq 2^{\sigma-1}CV(2.5)\int_{B(x,R+3)}g_u^{\sigma}d\mu.
\ee
Next to obtain an estimate on $(II)$ we choose $x_0\in Y$ such that $d(x,x_0)<1.$ Hence,
$$Y\cap B(x, R+1)\subset B(x_0,R+2).$$
To apply discrete Poincar\'e inequality on $Y$ we choose $r=V(4.5)(R+4)$ and  $h(r)=V(1)V(2r)$ . Therefore, using (\ref{ri2}) and (\ref{qi2}) we obtain,
$$Y\cap B_X(x_0,R+2)\subset B_Y(x_0,r) \ \  {\rm and} \ \ \nu(B_Y(x_0,r))\leq h(r).$$
Next we choose $\eta=\tilde{u}_r$ and estimate $(II)$ for this particular value of $\eta.$ From Theorem \ref{PGr2} we have,
$$(II)\leq 2^{2\sigma-1}r^{\sigma-1}h(r)\|\delta u\|^\sigma_{\sigma,B_Y(x_0,r)}.$$
As $B_Y(x_0,r)\subset B_X(x_0,2r)$,
$$(II)\leq 2^{2\sigma-1}r^{\sigma-1}h(r)\|\delta u\|^\sigma_{\sigma,B_X(x_0,2r)\cap Y}.$$
(\ref{IL}) implies that
\Bea (II)&\leq & 2^{3\sigma-2}C(1+V(3))^{2}V(6.5)V(1)r^{\sigma-1}V(2r)\int_{B(x_0,2r+3)}g_u^{\sigma}d\mu\\
&\leq & 2^{3\sigma-2}C(1+V(3))^{2}V(1)V(6.5)r^{\sigma-1}V(2r)\int_{B(x,2r+4)}g_u^{\sigma}d\mu
.\Eea
Since $R\geq 1$ from (\ref{I2}) we have,
\be \label{II2} (II) \leq  2^{3\sigma-2}C(1+V(3))^{3}V(6.5)r^{\sigma-1}V(2r)\int_{B(x,2r+4)}g_u^{\sigma}d\mu
.\ee
Let $\lambda=V(4.5)+1$ and $R\geq 4\lambda$ . Then putting the value of $r$ in (\ref{II2}) we have,
$$\int_{B(x,R)}|u-\tilde{u}_r|^{\sigma}d\mu\leq  2^{3\sigma-2}CV(6.5)\lambda^{3} R^{\sigma-1}V(2\lambda R)\int_{B(x,2\lambda R)}g_u^{\sigma}d\mu.$$
Now the required result follows from the following inequality.
\be
\int_{B(x,R)} |u-u_R|^{\sigma} \ d\mu \leq 2^{\sigma} \underset{\tau \in \mathbb{R}}{inf} \int_{B(x,R)} |u -\tau|^{\sigma} \ d\mu\leq \int_{B(x,R)}|u-\tilde{u}_r|^{\sigma}d\mu.
\ee
\end{proof}
It would be interesting to improve the Poincar\'e constants we obtained in Theorem \ref{main} and Theorem \ref{main2}. Moreover, it is not clear to us how the Poincar\'e constant changes if the ball of radius $2\lambda R$ on the right-hand side of the inequality is replaced by a ball of radius $R$ .
\section{Poincar\'e inequalities on measured metric spaces with polynomial growth} \label{secpoly}
If the growth function in (\ref{eqmain}) is a polynomial, then as a consequence of Theorem \ref{main} or Theorem \ref{main2}, the growth of the Poincar\'e constant is also polynomial. We discuss this case in this section. 
\begin{cor} Let $(M,g)$ be a complete Riemannian manifold with dimension $n$ and $Ric \geq -kg$ for some $k>0.$ If $\frac{Vol(B(x,R))}{Vol(B(x,\frac{1}{2}))}\leq V_0R^\alpha$ for all $R\geq r>0$ then for any $\sigma\geq 1$ there exist constants $C_0(n,k,r,\sigma,V_0,\alpha)>0$ and $\lambda (V_0,\alpha)\geq 1$ such that   
$$
\int_{B(x,R)} |u-u_R|^{\sigma} dv_g \leq C_0R^{\alpha+\sigma-1} \int_{B(x,\lambda R)} |\nabla u |^{\sigma}  dv_g, \ \ \ \forall u\in C^1(M), \forall R\geq r
$$
where $\nabla u$, $dv_g$ denote the gradient of $u$ and the volume form of $(M,g)$ respectively.
\end{cor}
\begin{proof} From Theorem 1.14 in \cite{HCB}, if $Ric\geq -kg$ then $(M,g)$ satisfies a local Poincar\'e inequality (\ref{local Poincare inequality}) and the Poincar\'e constant $C(n,k,R)$ depends on $k$, $n$ and $R$. Let $C=\sup_{R\leq 1}C(n,k,R).$ Then $(M,g)$ satisfies a Poincar\'e inequality as in (\ref{local Poincare inequality}). 
Now the result follows from Theorem \ref{main}. 
\end{proof}
The above result is proved in \cite{HCB} under an additional assumption of a positive lower bound on unit balls. In \cite{CK}, Croke and Karcher gave examples of complete Riemannian manifolds with positive Ricci curvature such that the infimum of unit balls is zero. The main theorem in \cite{HCB} does not hold in this case. More generally, let us consider a measured metric space $(X,d,\mu).$ 
\begin{definition}\label{doubling} A measure $\mu$ is called doubling for $r\geq r_1$ if there exists $C_0\geq 1$ such that 
$$
\frac{\mu(B(x,2r))}{\mu(B(x,r))}\leq C_0, \ \ \forall x\in X, \ \ \forall r\geq r_1.
$$
\end{definition}
$\mu$ is doubling if the growth of volume is polynomial.
\begin{cor}\label{sigma beta sigma} Let $(X,d,\mu)$ be a measured metric space. Suppose $(X,d,\mu)$ satisfies $P_{loc}$ for $r_0, \sigma \geq 1$ as in (\ref{local Poincare inequality}) and $\mu$ is doubling for $R>0$ with the doubling constant $C_0$. Then there exist positive constants $s(C_0)$, $C_1(\sigma,C_0,C)$ and $\lambda(C_0)\geq 1$ such that for any $u\in C(X)$ and its upper gradient $g_u$,
$$\int_{B(x,R)} |u - u_R|^{\sigma} \ d\mu \leq C_1 R^{\sigma+s-1}\int_{B(x,\lambda R)}g_u^{\sigma} d\mu, \ \ \forall R\geq 2\lambda,
	$$
	where $C$ is the constant term in $P_{loc}.$
\end{cor}
\begin{proof} 
If $\mu$ is doubling, then from Lemma 5.2.4 in \cite{SC},
$$
\frac{\mu(B(x,R))}{\mu(B(x,\frac{1}{2}))}\leq C_0^2R^s \ \ {\rm with} \ s=\frac{\log C_0}{\log 2}.
$$
for all $x\in X$ and $R \geq \frac{1}{2}$. Define,
\Bea f(R) &=& C_0^2R^s, \quad \forall R\geq \frac{1}{2};\\
          &=& 1\quad {\rm otherwise}. 
          \Eea
Then $\frac{\mu(B(x,R))}{\mu(B(x,\frac{1}{2}))}\leq f(R)$, for all $R>0$. Hence the result follows from Theorem \ref{main}. 
\end{proof}
G. Carron showed that a complete Riemannian manifold satisfying the doubling condition has finitely many ends in \cite{GC}. Riemannian manifolds satisfying $P_{loc}$ and the doubling condition also satisfy the parabolic Harnack inequality and Li-Yau type heat kernel estimates \cite{GIS}. The above corollary gives a curvature-free criterion on a Riemannian manifold with doubling measure to support a $(\sigma,\beta,\sigma)$-type uniform Poincar\'e inequality.  Some interesting examples of measured metric spaces which are not Riemannian manifolds but satisfy the criterion of Corolarry \ref{sigma beta sigma} are as follows. Topological manifolds with Ahlfors regular measures, which also satisfy a local contractibility condition, support $P_{loc}$ \cite{SS}. Carnot groups with Carnot-Carath\'eodory metrics satisfy $P_{loc}$ and a polynomial growth condition. We refer to Sections 10 and 11 in \cite{HK} for more details on these two types of examples.
\begin{cor} Let $(X,d,\mu)$ be a measured metric space which satisfies $P_{loc}$ for $r_0, \sigma\geq 1$ as in (\ref{local Poincare inequality}) and $\mu(B(x,\frac{1}{2}))\geq 1$. Suppose $\mu$ is doubling for $R\geq r_1$ with the doubling constant $C_0$ and $\mu(B(x,r_1))\leq V_0$, $\forall x\in X$. Then there exist positive constants $s(C_0)$, $C_1(\sigma,C_0,C,V_0)$ and $\lambda(C_0,V_0)\geq 1$ such that for any $u\in C(X)$ and its upper gradient $g_u$,
$$\int_{B(x,R)} |u - u_R|^{\sigma} \ d\mu \leq C_1 R^{\sigma+s-1}\int_{B(x,\lambda R)}g_u^{\sigma} d\mu \ \ \forall R\geq 2\lambda,
$$
where $C$ is the constant term in $P_{loc}.$
\end{cor}
\begin{proof} 
If $\mu$ is doubling then for all $R\geq r$ and $x\in X$ from Lemma 5.2.4 in \cite{SC},
\be 
\frac{\mu(B(x,R))}{\mu(B(x,r))}\leq C_0\left(\frac{R}{r}\right)^s \ \ {\rm with} \ s=\frac{\log C_0}{\log 2}.
\ee
\Bea {\rm Define} ,\ \ f(R)&=& V_0, \ \ \forall \ \frac{1}{2}\leq R \leq r_1;\\
  &=& \frac{V_0}{r^s} R^s ,\ \ \forall R>r_1.
\Eea 
Now the required result is an immediate consequence of Theorem \ref{main2}. 
\end{proof}
Hajlasz and Koskela established a different type of global Poincar\'e inequality on metric spaces with doubling measures, which satisfies a chain condition $C(\lambda,M)$ for $\lambda$, $M\geq 1$ in \cite{HK2}. They also showed that every ball in a geodesic metric space satisfies this chain condition for every $\lambda\geq 1.$  Now, consider a subset $\Omega$ on a geodesic metric space with a doubling measure, which satisfies the chain condition  $C(\lambda,M)$ for some $\lambda$, $M\geq 1$ . From Theorem 1 in \cite{HK2}, $\Omega$ supports a Poincar\'e inequality in the sense of \cite{HK2} if a Poincar\'e inequality holds on every ball $B$ with $\lambda B\subset \Omega$. If  $\Omega$ contains a  ball $B$ with a large radius such that $\lambda B\subset \Omega$, then $B$  should necessarily support a Poincar\'e inequality for the conclusion to hold. Whereas, in our case, if $P_{loc}$ holds on all unit balls, then a Poincar\'e inequality holds on any ball with a sufficiently large radius as stated in Corollary \ref{sigma beta sigma}. One can also describe the Poincar\'e constant more precisely using Theorem \ref{main}. 

A more general notion doubling condition $DV_{loc}$ is used in \cite{SLT} and \cite{Ch} to prove various analytic and geometric properties of measured metric spaces. $(X,d,\mu)$ satisfies $DV_{loc}$ if for any $r>0$ there exists a constant $C(r)>0$ such that 
$$
\frac{\mu(B(x,2r))}{\mu(B(x,r))}\leq C(r), \ \ \forall x\in X.
$$ 
We observe that $(X,d,\mu)$ satisfies $DV_{loc}$ for $r\geq \frac{1}{2}$ if and only if there exists $f:(0,\infty)\to (0,\infty)$ such that (\ref{eqmain}) holds. Next we study Poincar\'e inequalities on $(X,d,\mu)$ when $f(r)$ grows exponentially in Section \ref{sec4}. 
\section{Poincar\'e inequality on Gromov hyperbolic spaces}\label{sec4} 
Complete simply connected Riemannian manifolds with negative sectional curvature have  exponential growth of volume. Gromov-hyperbolic spaces are generalizations of them. A Gromov hyperbolic space is a $\delta$-hyperbolic space for some $\delta\geq 0.$ $\delta$-hyperbolic spaces are defined in various ways, but the value of $\delta$ changes if we change the definition. As theorems in this section depend on the precise value of $\delta$, we consider the following definition of $\delta$-hyperbolic spaces in this paper. 
\begin{definition} Let $(X,d)$ be a proper geodesic metric space. A geodesic triangle $T$ is called $\delta$-thin if any point on one of its sides is contained in the $\delta$-neighbourhood of the union of the other two sides. $(X,d)$ is called a $\delta$-hyperbolic space if all of its geodesic triangles are $\delta$-thin. 
\end{definition} 
Let $\Gamma$ be a finitely generated group with a finite set of generators $S$. $(\Gamma,S)$ is called a hyperbolic group if the Cayley graph $\Gamma$ of $(\Gamma,S)$ is $\delta$-hyperbolic as a metric space. If a discrete group $\Gamma$ acts on a $\delta$-hypebolic space properly and co-compactly, then $\Gamma$ is a hyperbolic group.
\begin{thm}\label{Phy} Let $(X,d,\mu)$ be a measured $\de$-hyperbolic space which supports $P_{loc}$ for some $r_0, \sigma\geq 1$ as in (\ref{local Poincare inequality}). Let $\Gamma$ be a discrete group acting on $X$ isometrically and properly such that the diameter of the quotient space $\Gamma \backslash X$ is bounded by $D\geq 1$. Suppose the action of $\Gamma$ is measure preserving, the entropy of $(X,d,\mu)$ is bounded by $H$ and $\mu(B(x,\frac{1}{2}))\geq 1$ for all $x\in X$. Then there exist $C_0(\delta,D,H,\mu)>0$ and $\lambda \geq 1$ such that for any $u\in C(X)$ and its upper gradient $g_u$,
$$\int_{B(x,R)} |u - u_R|^{\sigma} d\mu \leq  \frac{2^{4\sigma+6HD+\frac{25}{4}}\lambda^{\frac{41}{4}+6HD} CV_0}{5^{6HD+\frac{25}{4}}r^{6HD+\frac{25}{4}}e^{H(12r-D)}}R^{\sigma+6HD+\frac{21}{4}}e^{12\lambda HR}\int_{B(x,2\lambda R)}g_u^{\sigma} d\mu,   $$
$\forall R\geq \frac{5r}{2}$ and $\forall z\in X$ where $r=7D+4\delta$, $V_0=\sup_{x\in X}\{\mu(B(x,\frac{5r}{2}))\}$ and  $\lambda=V_0+1$.
\end{thm}
\begin{proof} Let $V_0=\sup\{\mu(B(x,\frac{5r}{2}))|x\in X\}$. Since the action of $\Gamma$ is co-compact, $V_0$ is finite. From Theorem 1.9 part (i) in \cite{BCS} we have,
\Bea \mu(B(x,R))\leq \left(\frac{2}{5}\right)^{\frac{25}{4}+6HD}\frac{3V_0R^{\frac{25}{4}+6HD}e^{6HR}}{r^{\frac{25}{4}+6HD}e^{H(12r-D)}}.\Eea
Define,
\Bea V(R)&=& V_0, \quad \forall \ R< \frac{5r}{2};\\
         &=& \left(\frac{2}{5}\right)^{\frac{25}{4}+6HD}\frac{3V_0R^{\frac{25}{4}+6HD}e^{6HR}}{r^{\frac{25}{4}+6HD}e^{H(12r-D)}}, \quad \forall \ R\geq \frac{5r}{2}.
\Eea
Since $D\geq 1$, $r\geq 7$, $\frac{5r}{2}\geq 6.5$. Hence $V(4.5)=V(6.5)=V_0.$ Now from Theorem \ref{main2} for all $R\geq \frac{5r}{2}$,
\Bea \int_{B(x,R)} |u(z) - u_R|^{\sigma} d\mu(z)&\leq& 2^{4\sigma-2}C\lambda^{2\sigma+1} R^{\sigma-1}V(2\lambda R)\int_{B(x,2\lambda R)}g_u(z)^{\sigma} d\mu(z) \\
 &\leq & \frac{2^{4\sigma+6HD+\frac{25}{4}}\lambda^{2\sigma+\frac{33}{4}+6HD} CV_0}{5^{6HD+\frac{25}{4}}r^{6HD+\frac{25}{4}}e^{H(12r-D)}}R^{\sigma+6HD+\frac{21}{4}}e^{12\lambda HR}\int_{B(x,2\lambda R)}g_u(z)^{\sigma} d\mu(z)   . 
\Eea
\end{proof}
Since $\Gamma\backslash X$ is compact $\inf_{x\in X} \mu(B(x,\frac{1}{2}))>0.$ If this quantity is less than 1, one can scale the metric suitably to obtain the required lower bound. Clearly, if $H$ is zero then the Poincar\'e constant in the above theorem is a polynomial in $R$ and $(X,d,\mu)$ satisfies a $(\sigma,\beta,\sigma)$-type uniform Poincar\'e inequality as stated in Section \ref{secpoly}. $X$ is called elementary if the boundary of $X$ contains at most two elements. If $X$ is elementary, then the entropy of the space is zero (see Proposition 8.43 in \cite{BCS}). Therefore the Poincar\'e constant of an elementary $\delta$-hyperbolic space grows polynomially. We recall Proposition 1.3 from \cite{BCS} on non-elementary $\delta$-hyperbolic spaces.
\begin{prop}\cite{BCS} For every non-elementary $\delta$-hyperbolic metric space $(X,d)$ and for every group $\Gamma$ acting on it properly by isometries, if the diameter of $\Gamma \backslash X \leq D<\infty$ then, 
$$Ent(X,d,\mu)\geq \frac{ln2}{27\delta+10D}$$
where $\mu$ is any measure on $(X,d)$ preserved by the action of $\Gamma$.   
\end{prop}
Therefore, when $(X,d)$ is non-elementary the growth of the Poincar\'e constant is exponential. Theorem \ref{Phy} only gives an upper bound of the best constant for which a global Poincar\'e inequality on $(X,d,\mu)$ holds. It would be interesting to know if the best Poincar\'e constant also grows exponentially in this case.
 
If a discrete group acts on a Riemannian manifold $(M,g)$ as mentioned in Theorem \ref{Phy}, then the Ricci curvature of $M$ is bounded. Hence $(M,g)$ supports $P_{loc}$. Consequently, $(M,g)$ satisfies a global Poincar\'e inequality as stated in Theorem \ref{Phy}. A large class of compact singular Riemannian spaces with almost smooth metrics also support a local Poincar\'e inequality. K. Akutagawa, G. Carron, and R. Mazzeo showed that compact stratified spaces with iterated edge metrics satisfy local Poincar\'e and sobolev inequality \cite{ACM}. We refer to \cite{ACM2}, \cite{BKMR} for a detailed geometric description of stratified spaces, examples, and more analytic properties on them. If the quotient space $\Gamma\backslash X$ in Theorem \ref{Phy} is a stratified space then $X$ support a $P_{loc}$ by Theorem \ref{PCov}. Hence $X$ satisfies a global Poincar\'e inequality. Consequently, $\delta$-hyperbolic polyhedral spaces with piece-wise smooth Riemannian metric admitting a co-compact group action as stated in Theorem \ref{Phy} support a global Poincar\'e inequality.

When $(X,d)$ is a $\delta$-hyperbolic graph an assumption on local Poincar\'e inequality is not required.
\begin{cor} Let $\Gamma$ be a $\delta$ hyperbolic Cayley graph of a hyperbolic group $G$ equipped with a measure $\mu$. Let $|g_u|_\sigma$ denote the point-wise $l^{\sigma}$-norm of the gradient of a function $u$ on $\Gamma$. Suppose $c\leq \mu(x)\leq C$ for all $x\in \Gamma$ and the entropy of $(\Gamma,\mu)$ is bounded by $H$. Then for $R\geq r=10(1+\delta)$, $\sigma\geq 1$ and $u:\Gamma\to \mathbb{R}$
$$\int_{B(p,R)}|u-u_R|^\sigma d\mu\leq \frac{2^{\sigma+2} C\nu(r)}{cr^{\frac{25}{4}}e^{48H(1+\delta)}}R^{\sigma+\frac{21}{4}}e^{6HR}\int_{B(p,R)}|g_u|_{\sigma} ^{\sigma}d\mu, \ \ \forall p\in X$$
where $\nu(r)$ is the number of elements of $\Gamma$
in a ball of radius $r$.
\end{cor}
\begin{proof}
Let $\nu$ be the counting measure on $\Gamma.$ Then $\nu$ is a $G$-invariant measure on $\Gamma$. Hence, for any $R>0$, $\nu(B(x,R))$ is same for all $x\in \Gamma.$ We define $\nu(R)=\nu(B(x,R)).$ $G$ acts on the set of vertices of $\Gamma$ transitively. Hence the diameter of $G\backslash \Gamma$ bounded above by $1$. Therefore, from Theorem 1.9 part (ii) in \cite{BCS} we have,
$$\nu(R)< 3\nu(r)\left(\frac{R}{r}\right)^{\frac{25}{4}}e^{6H(R-\frac{4r}{5})}, \ \ \forall \ R\geq r=10(1+\delta).$$
Since $\mu(B(x,R))\leq C\nu(R)$, applying Theorem \ref{PGr2} for all $R\geq r$ we have,
\Bea \int_{B(p,R)}|u(x)-u_R|^\sigma d\mu(x)&\leq &  \frac{2^{\sigma}.3C\nu(r)}{cr^{\frac{25}{4}}}R^{\sigma+\frac{21}{4}}e^{6H(R-\frac{4r}{5})}\int_{B(p,R)}|g_u|_{\sigma}^{\sigma}(x)d\mu(x) \\
&\leq & \frac{2^{\sigma+2 }C\nu(r)}{cr^{\frac{25}{4}}e^{48H(1+\delta)}}R^{\sigma+ \frac{21}{4}}e^{6HR}\int_{B(p,R)}|g_u|_{\sigma}^{\sigma}(x)d\mu(x).
\Eea
Hence the proof follows.
\end{proof} 
Examples of hyperbolic groups are fundamental groups of negatively curved manifolds, finitely generated free groups, or any discrete group acting on a $\delta$-hyperbolic space properly and isometrically such that the quotient space is compact. Hyperbolic groups play an essential role in geometric group theory \cite{BH}. A general $\delta$-hyperbolic graph satisfies the following Poincar\'e inequality.
\begin{cor}\label{HyGr} Let $(X,d,\mu)$ be a $\delta$-hyperbolic graph.  Let $\Gamma$ be a discrete group acting on $X$ isometrically and properly such that the diameter of the quotient space $\Gamma \backslash X$ is bounded by $D\geq 1$. Let $|g_u|_\sigma$ denote the point-wise $l^{\sigma}$-norm of the gradient of a function $u$ on $\Gamma$. Suppose the action of $\Gamma$ is measure preserving, and the entropy of $(X,d,\mu)$ is bounded by $H$. Then For any $u\in C(X)$, $\forall R\geq \frac{5r}{2}$ and $\forall z\in X$,
$$\int_{B(x,R)} |u - u_R|^{\sigma} d\mu \leq \frac{2^{\sigma+8+6HD}cV_0}{5^{6(1+HD)}r^{\frac{25}{4}+6HD}e^{H(12r-1)}}R^{\sigma+\frac{21}{4}+6HD}e^{6HR} \int_{B(x,R)}|g_u|_{\sigma}^{\sigma}d\mu$$
where $r=7D+4\delta$, $V_0=\sup_{x\in X}\{\mu(B(x,\frac{5r}{2}))\}$, and $\frac{1}{c}=\inf_{x\in X} \mu(x)$.
\end{cor}
\begin{proof} Let $V_0=\sup\{\mu(B(x,\frac{5r}{2}))|x\in X\}$. Since the action of $\Gamma$ is co-compact $V_0$ is finite. From Theorem 1.9 part (i) in \cite{BCS} we have, for all $R\geq \frac{5r}{2}$,
\Bea \mu(B(x,R))\leq \left(\frac{2}{5}\right)^{\frac{25}{4}+6HD}\frac{3V_0R^{\frac{25}{4}+6HD}e^{6HR}}{r^{\frac{25}{4}+6HD}e^{H(12r-D)}}.\Eea
Define, 
\Bea f(R) &=& V_0, \quad  \forall R\leq \frac{5r}{2};\\
          &=& \left(\frac{2}{5}\right)^{\frac{25}{4}+6HD}\frac{3V_0R^{\frac{25}{4}+6HD}e^{6HR}}{r^{\frac{25}{4}+6HD}e^{H(12r-D)}} , \quad \forall R>\frac{5r}{2}.
          \Eea
Since $\Gamma \backslash X$ is compact and the action of $\Gamma$ is measure preserving, $\inf_{x\in X} \mu(x)>0.$ Now,  as a consequence of Theorem \ref{PGr2} we have,
\Bea \int_{B(p,R)}|u(x)-u_R|^\sigma d\mu(x)&\leq & 2^{\sigma}c R^{\sigma-1}f(R)\int_{B(p,R)}|g_u|_{\sigma}^{\sigma}(y)d\mu(y) \\
&\leq & \frac{2^{\sigma+8+6HD}cV_0}{5^{6(1+HD)}r^{\frac{25}{4}+6HD}e^{H(12r-1)}}R^{\sigma+\frac{21}{4}+6HD}e^{6HR} \int_{B(p,R)}|g_u|_{\sigma}^{\sigma}(y)d\mu(y).
\Eea
\end{proof}
\section{Uniform Poincar\'e inequalities and growth of groups}\label{sec5} 
Relations between the growth of groups, volume, and curvature motivate us to understand the dependence of the growth of Poincar\'e constants on the growth of groups. Let $\Gamma$ be a discrete subgroup of isometries of a measured metric space $(X,d,\mu)$ acting on it freely and properly such that $\Gamma\backslash X$ is compact. The point-wise systole is defined as $sys_{\Gamma}(x)=\inf_{\gamma\neq e} d(x, \gamma.x).$ The systole $sys_{\Gamma}$ of $X$ is the infimum of $sys_{\Gamma}(x)$ over $x\in X$. Since $\Gamma \backslash X$ is compact $sys_{\Gamma}$ is non-zero. Define the quotient metric $\bar{d}$ on $\Gamma\backslash X$ as follows. For $x,y\in \Gamma\backslash X$ choose $\tilde{x}\in p^{-1}(x)$ and $\tilde{y}\in p^{-1}(y).$ Then
\be 
\bar{d}(x,y)=\inf_{\gamma\in \Gamma}d(\tilde{x},\gamma.\tilde{y})= \inf_{\gamma_1,\gamma_2\in \Gamma}d(\gamma_1.\tilde{x},\gamma_2.\tilde{y}).\ee
The quotient map $p$ is a covering map and  $p:B(\tilde{x},\frac{sys_{\Gamma}}{4})\to B(x,\frac{sys_{\Gamma}}{4})$ is an isometry for all $x\in \Gamma\backslash X$ and $\tilde{x}\in p^{-1}(x)$. Suppose the action of $\Gamma$ is measure-preserving. Define, the quotient Borel measure $\bar{\mu}$ on $X\backslash \Gamma$ such that $p$ restricted to each $B(\tilde{x},\frac{sys_{\Gamma}}{4})$ is measure-preserving for all $\tilde{x} \in X$.
\begin{lem} \label{weak inequality} Suppose $(X\backslash \Gamma, \bar{d},\bar{\mu})$ satisfies a Poincar\'e inequality, i.e. for a fixed $\sigma\geq 1$ there exist a constant $C(R)>0$ such that for any $u\in C(X)$ and its upper gradient $g_u$,
\be\label{weakpoincare}\int_{B(x,R)} |u-u_R|^{\sigma} d\bar{\mu} \leq C(R)\int_{B(x,R)} g_u^{\sigma} d\bar{\mu}, \quad \forall R\in (0,\frac{sys_{\Gamma}}{4}] 
\ee
then $(X,d,\mu)$ satisfies $P_{loc}$ for any $R \leq \frac{sys_{\Gamma}}{4}$.
\end{lem}
\begin{proof} Consider a continuous function $\tilde{u}:X \to \mathbb{R}$ and an upper gradient $g_{\tilde{u}}$ of $\tilde{u}$.  Given $R <\frac{sys_{\Gamma}}{4} $ choose $R'$ such that  $R< R' < \frac{sys_{\Gamma}}{4}$ and a bump function $\phi : X \to \mathbb{R}$ such that $\phi =1$ on $B(\tilde{x},R)$ and $\phi=0$ outside $B(\tilde{x},R')$. Define, $u:\Gamma\backslash X\to \mathbb{R}$ as 
$$u (y) = \sum_{\tilde{y}\ \in\ p^{-1}(y)}\tilde{u}(\tilde{y}) \phi(\tilde{y}).$$
Corresponding to the upper gradient $g_{\tilde{u}}$ we can define an upper gradient for $u$ on $B(x,R)$ as
$$g_u (y) = \sum_{\tilde{y}\ \in\ p^{-1}(y)}g_{\tilde{u}}(\tilde{y}) \phi(\tilde{y}).$$
To see that $g_u$ is indeed an upper gradient for $u$ on $B(x,R)$ choose a unit speed curve $\alpha : [0,1] \to B(x,R)$ and let $\tilde{\alpha}$ be its lift via the map $p$ passing through $B(\tilde{x},R)$. Then,
	$$|u(\alpha(1))-u(\alpha(0))| = |u(\tilde{\alpha}(1))-u(\tilde{\alpha}(0))| \leq \int_{0}^{1} g_{\tilde{u}} (\tilde{\alpha}(t)) dt = \int _{0}^{1} g_{u}(\alpha(t)) dt .$$
Let $C=\sup \{C(R): R\leq \frac{sys_{\Gamma}}{4}\} $. Now we have the required local Poincar\'e inequality as
	$$\ \int_{B(\tilde{x},R)} |\tilde{u}-\tilde{u}_R|^{\sigma} d\mu= \int_{B(x,R)} |u-u_R|^{\sigma} d\bar{\mu} \leq C\int_{B(x,R)} g_u^{\sigma} d\bar{\mu} = C\int_{B(\tilde{x},R)} |g_{\tilde{u}}|^{\sigma} d\mu .$$
	\end{proof} 
Suppose $\Gamma$ is a discrete group acting on $(X,d)$ freely, properly, and isometrically. Then recall that $$F_{\Gamma}(R)=|\Gamma x\cap \overline{B(x,R)}|.$$ 	
	
\begin{thm}\label{PCov} Consider a measured metric space $(X,d,\mu).$ Let $\Gamma$ be a discrete subgroup of isometries of $(X,d,\mu)$ acting on it freely and  properly such that $\Gamma\backslash X$ is compact and $sys_{\Gamma}\geq 4$. Suppose $\mu(B(x,\frac{1}{2}))\geq 1$ for all $x\in X$. Let the diameter and the volume of $\Gamma\backslash X$ be bounded above by $D$ and $V_0$, respectively. If $(\Gamma\backslash X,\bar{d},\bar{\mu})$ satisfies a Poincar\'e inequality (\ref{weakpoincare}) then for any $u\in C(X)$ and its upper gradient $g_u$,
	\be
	\int_{B(x,R)} |u-u_R|^{\sigma} \ d\mu \leq 2^{4\sigma-2}CV_0\lambda^{4} R^{\sigma-1}F_{\Gamma}\left(2\lambda R
	\right)\int_{B(x,2\lambda R)} g_u^{\sigma} \ d\mu
	\ee
	for all $R \geq 4\lambda$ and $\sigma \geq 1$, where $\lambda = V_0+D$. 
\end{thm}
\begin{proof} Since the quotient space $\Gamma\backslash X$ supports a Poincar\'e inequality by  Lemma \ref{weakpoincare}, $X$ satisfies $P_{loc}$ as defined in (\ref{local Poincare inequality}). Since the covering map is locally measure preserving for any $R>0$,
\be \label{volume} \mu(B(x,R))\leq V_0F_{\Gamma}(R+D), \ \ \forall \ x\in X.
\ee
For any $R\leq \frac{sys_{\Gamma}}{4}$ and $x\in X$, $\mu(B(x,R))\leq V_0$. Define,
\Bea f(R)&=& V_0, \quad \forall \ R\leq \frac{sys_{\Gamma}}{4}\\
&=& V_0F_{\Gamma}(R+D), \quad \ \forall \ R> \frac{sys_{\Gamma}}{4}.
\Eea
Now the theorem follows from Theorem \ref{main2}.
	\end{proof}
The required lower bound on the systole and $\mu(B(x,\frac{1}{2}))$ may be achieved by scaling $\bar{d}$ suitably. 
\subsection*{Proof of Theorem \ref{CDK} :} If the quotient space $\Gamma\backslash X$ is a $CD(K,\infty)$ space then $\Gamma\backslash X$ supports a Poincar\'e inequality by Proposition \ref{Rajala}. Now the proof of Theorem \ref{CDK} follows from Theorem \ref{PCov}.
   
\begin{cor}\label{UPR} Let $(M,g)$ be a complete noncompact Riemannian manifold with dimension $n$. Let $\Gamma$ be a discrete subgroup of isometries of $(M,g)$ acting on it properly such that the diameter of the quotient space $\Gamma\backslash M$ is bounded by $D$. Suppose $Vol(B(x,\frac{1}{2}))\geq 1$ and sectional curvatures of $(M,g)$ are bounded below and above by $\kappa$ and $K$ respectively. Then there exist positive constants $C(n,\kappa, K,\sigma)$ and $\lambda(n,\kappa,K)\geq 1$ such that for any $u\in C^1(M)$ and $R\geq 4\lambda$,
$$\int_{B(x,R)} |u - u_R|^{\sigma} \ dv_g \leq C R^{\sigma-1}F(2\lambda R)\int_{B(x,2\lambda R)}| \nabla u|^{\sigma} dv_g, \quad \forall \ x\in X $$
where $dv_g$ is the volume form induced from $g.$ 
  
\end{cor}
\begin{proof} Consider a Riemannian manifold $(M,g)$ and a group $\Gamma$ which satisfy the assumptions  of Corollary \ref{UPR}.   Then for any $x\in M$, $\{B(\gamma.x,D)\}_{\gamma\in \Gamma}$ covers $M$. By continuity of Riemannian curvature, sectional curvatures of $(M,g)$ are bounded on each $\overline{B(\gamma. x,D)}$. Since $\Gamma$ acts isometrically, sectional curvatures of $(M,g)$ are bounded. Let $\kappa \leq sec \leq K.$ From Theorem 1.14 in \cite{HCB} there exists a constant $C(n,\kappa,R)>0$  such that for any $u\in C^1(M)$,
$$\int_{B(x,R)} |u-u_R|^{\sigma} dv_g \leq C(n,\kappa,R)\int_{B(x,R)} |\nabla u|^{\sigma} dv_g, \ \ \forall \ R>0, 
\ \forall \ x\in M.$$
Let $C=\sup_{0<R\leq 1}C(n,\kappa,R).$ Hence $(M,g)$ satisfies $P_{loc}$ as in (\ref{local Poincare inequality}).  Let $Vol(B(x,R))$ denote the volume of $B(x,R)$. Since $\Gamma$ acts isometrically,
$$\sup_{x\in M}Vol(B(x,R))= \sup\{Vol(B(y,R)):y\in \overline{B(x,D)}\}.$$ 
Hence it is bounded above. Similarly $\inf_{x\in M}Vol(B(x,R))$ is also bounded below by a positive constant for any $R>0.$ From Lemma 3.5 in \cite{BCS} for any $R>0,$
$$Vol(B(x,R))\leq Vol(B(x,D))F_{\Gamma}(R+D).$$ 
Let $V^K_n(R)$ denote the volume of a  ball of radius $R$ in the space of constant curvature $K$. From the Bishop-Gromov volume comparison theorem $Vol(B(x,D))\leq V_n^\kappa(D)$ and $Vol(B(x,\frac{1}{2}))\geq V_n^K(\frac{1}{2})$ for all $x\in M.$ Hence,
\Bea Vol(B(x,R))\leq V_n^{\kappa}(D) F_{\Gamma}(R+D).
\Eea
Now the required result follows from Theorem \ref{main2}.
\end{proof} 
When $M$ is simply connected, Corollary \ref{UPR} shows the dependence of the growth of the Poincar\'e constant on the growth of the fundamental group of the quotient space. If $(M,g)$ is the universal cover of a compact Riemannian manifold with non-negative Ricci curvature, then the growth of volume and the growth of the fundamental group $\Gamma$, is polynomial \cite{Mi}. Hence, $(M,g)$ satisfies a $(\sigma,\beta,\sigma)$-uniform Poincar\'e inequality as in Corollary 3.2. If $(M,g)$ is the universal cover of a compact negatively curved Riemannian manifold, then the growth of the fundamental group $\Gamma$ is exponential \cite{Mi}. 

More generally, consider a simply connected measured metric space $(X,d,\mu)$. As a consequence of Theorem \ref{PCov}, the growth of the Poincar\'e constant with respect to $R$ is polynomial (or exponential) if the growth of $\Gamma$ is polynomial (or exponential). Given a fixed measured metric space $(X,d,\mu)$,  multiple groups may act isometrically, properly on $X$,  preserving the measure $\mu$. In this case, the group with minimal growth gives a better Poincar\'e constant. It would be interesting to find out how the best Poincar\'e constant of $(X,d,\mu)$ is related to the growth of a group with minimal growth. 
\subsection*{Conflict of Interests statement: }  On behalf of all authors, the corresponding author states that there is no conflict of interest.

\end{document}